\documentclass[11pt]{amsart} %set font size up here (in [])
\usepackage{amssymb,amsthm,amsmath,epsfig,latexsym}
\usepackage{graphicx}
\usepackage{calc,times}
\usepackage{geometry} %geometry package allows easy manipulation of margins, etc.
\begin{document}

% \mmbox enables macros to survive outside of $ ... $
\newcommand{\mmbox}[1]{\mbox{${#1}$}}
\newcommand{\proj}[1]{\mmbox{{\mathbb P}^{#1}}}
\newcommand{\Cr}{C^r(\Delta)}
\newcommand{\CR}{C^r(\hat\Delta)}
\newcommand{\affine}[1]{\mmbox{{\mathbb A}^{#1}}}
\newcommand{\Ann}[1]{\mmbox{{\rm Ann}({#1})}}
\newcommand{\caps}[3]{\mmbox{{#1}_{#2} \cap \ldots \cap {#1}_{#3}}}
\newcommand{\N}{{\mathbb N}}
\newcommand{\Z}{{\mathbb Z}}
\newcommand{\R}{{\mathbb R}}
\newcommand{\sat}{{\rm sat}}
\newcommand{\Tor}{\mathop{\rm Tor}\nolimits}
\newcommand{\Ext}{\mathop{\rm Ext}\nolimits}
\newcommand{\Hom}{\mathop{\rm Hom}\nolimits}
\newcommand{\im}{\mathop{\rm Im}\nolimits}
\newcommand{\rank}{\mathop{\rm rank}\nolimits}
\newcommand{\supp}{\mathop{\rm supp}\nolimits}
\newcommand{\arrow}[1]{\stackrel{#1}{\longrightarrow}}
\newcommand{\CB}{Cayley-Bacharach}
\newcommand{\coker}{\mathop{\rm coker}\nolimits}
\sloppy
\newtheorem{defn0}{Definition}[section]
\newtheorem{prop0}[defn0]{Proposition}
\newtheorem{quest0}[defn0]{Question}
\newtheorem{thm0}[defn0]{Theorem}
\newtheorem{lem0}[defn0]{Lemma}
\newtheorem{corollary0}[defn0]{Corollary}
\newtheorem{example0}[defn0]{Example}
\newtheorem{remark0}[defn0]{Remark}

\newenvironment{defn}{\begin{defn0}}{\end{defn0}}
\newenvironment{prop}{\begin{prop0}}{\end{prop0}}
\newenvironment{quest}{\begin{quest0}}{\end{quest0}}
\newenvironment{thm}{\begin{thm0}}{\end{thm0}}
\newenvironment{lem}{\begin{lem0}}{\end{lem0}}
\newenvironment{cor}{\begin{corollary0}}{\end{corollary0}}
\newenvironment{exm}{\begin{example0}\rm}{\end{example0}}
\newenvironment{rem}{\begin{remark0}\rm}{\end{remark0}}

\newcommand{\defref}[1]{Definition~\ref{#1}}
\newcommand{\propref}[1]{Proposition~\ref{#1}}
\newcommand{\thmref}[1]{Theorem~\ref{#1}}
\newcommand{\lemref}[1]{Lemma~\ref{#1}}
\newcommand{\corref}[1]{Corollary~\ref{#1}}
\newcommand{\exref}[1]{Example~\ref{#1}}
\newcommand{\secref}[1]{Section~\ref{#1}}
\newcommand{\remref}[1]{Remark~\ref{#1}}
\newcommand{\questref}[1]{Question~\ref{#1}}

\newcommand{\std}{Gr\"{o}bner}
\newcommand{\jq}{J_{Q}}

%\parskip = 4pt

%\begin{singlespace}
\title{Equivalent matrices up to permutations}

\author{\c{S}tefan O. Toh\v{a}neanu and Jesus Vargas}

 \keywords{permutation matrices, McEliece cryptosystem, Reed-Solomon codes, codewords of minimum weight.\\ \indent Authors' addresses: Department of Mathematics, University of Idaho, Moscow, ID 83844, tohaneanu@uidaho.edu, varg1778@vandals.uidaho.edu.}

\begin{abstract}
\noindent Given two $k\times n$ matrices $A$ and $B$, we describe a couple of methods to solve the matrix equation $XA=BY$, where $X$ is an invertible $k\times k$ matrix, and $Y$ is an $n\times n$ permutation matrix, both of which we want to determine. We are interested in pursuing those techniques that have algebraic geometric flavour. An application to solving such a matrix equation comes from the cryptanalysis of McEliece cryptosystem. By using codewords of minimum weight of a linear code, in concordance with these methods of solving $XA=BY$, we present an efficient way to determine the entire encryption keys for the McEliece cryptosystems built on Reed-Solomon codes.
\end{abstract}
\maketitle

\section{Introduction}\label{Introduction}

A {\em permutation} on the set $[n]:=\{1,\ldots,n\}$, is a bijective function $\sigma:[n]\rightarrow[n]$. The set of all permutations on the set $[n]$ forms a group, denoted here with $\mathbb S_n$, with multiplication being defined by the composition of functions. The identity element will be denoted here with $e$. For more information we suggest \cite[Chapter 5]{Ga}.

The most classical notation for a permutation $\sigma\in\mathbb S_n$ is $\sigma=\left(\begin{array}{cccc}1&2&\cdots&n\\ \sigma(1)&\sigma(2)&\cdots&\sigma(n)\end{array}\right)$. A {\em cycle} $(\,i_1i_2\cdots i_{m-1}i_m\,)$ is the permutation $\sigma$ defined as $\sigma(j)=j$ for all $j\in[n]\setminus\{i_1,\ldots,i_m\}$, and $\sigma(i_1)=i_2,\ldots,\sigma(i_{m-1})=i_m, \sigma(i_m)=i_1$. Every permutation is a product of disjoint cycles, and every two disjoint cycles commute. The inverse of the cycle $\sigma=(\,i_1i_2\cdots i_{m-1}i_m\,)$, is the cycle $\sigma^{-1}=(\,i_mi_{m-1}\cdots i_2i_1\,)$.

An $n\times n$ matrix $P$ is called a {\em permutation matrix} if there is $\sigma\in\mathbb S_n$ such that $P$ is obtained by permuting according to $\sigma$ the columns of the identity matrix $I_n$; if $\sigma(i)=j$, then the $i-$th column of $I_n$ is moved to become the $j-$ column of $P$. To specify the permutation $\sigma$, we denote here $P=I_n(\sigma)$. We have
$$(I_n(\sigma))^{-1}=I_n(\sigma^{-1})=(I_n(\sigma))^T.$$

\medskip

Let $\mathbb K$ be any field. Let $A$ and $B$ be two given $k\times n$ matrices with entries in $\mathbb K$. We say that $A$ and $B$ are {\em equivalent up to a permutation} if an only if there exist an invertible $k\times k$ matrix $S$ with entries in $\mathbb K$, and an $n\times n$ permutation matrix $P$ such that $$A=SBP.$$ Equivalently, the matrix equation $$XA=BY$$ has a solution (i.e., $X=S^{-1}$ and $Y=P$).

There are several instances where pairs of equivalent matrices up to a permutation occur:

\begin{itemize}
  \item Permutation matrices show up often in various matrix decompositions, especially in LU decomposition. For example, \cite[Theorem 1]{Jef} says that any $k\times n$ matrix $A$ of rank $r$ can be decomposed as $A=QBP$, where $B$ is a product of a $k\times r$ lower-triangular matrix and an $r\times n$ upper-triangular matrix, and both $Q$ and $P$ are permutation matrices of appropriate sizes.
  \item Suppose $G$ is the generating matrix (of size $k\times n$) of a linear code of dimension $k$. Using Gaussian elimination, we can bring $G$ to the ``standard form'', which is a matrix of the form $[I_k|A]$, where $I_k$ is the identity $k\times k$ matrix, and $A$ is a $k\times (n-k)$ matrix. While the Gaussian elimination process is captured in a $k\times k$ invertible matrix $S_G$, there exists an $n\times n$ permutation matrix $P_G$, such that $$S_GGP_G=[I_k|A];$$ just permute the pivot columns accordingly. The $(n-k)\times n$ matrix $H:=[-A^T|I_{n-k}]$ is called {\em the parity-check matrix} of the linear code. If we are to be picky, this is true only if $P_G$ is the identity matrix $I_n$. In general case, the true parity-check matrix of the code is $HP_G^T$.
  \item The public key of the McEliece cryptosystem is a matrix $A$ that it is known to be equal to $SBP$, where $S$ is invertible, $P$ is a permutation matrix, and $B$ is a convenient matrix of the same size as $A$, all these three matrices being a part of the encryption key. Though it would defeat the security purpose of this cryptosystem, if $B$ is also known, it is an interesting exercise to find the remaining two matrices $S$ and $P$. As an application to our methods of solving the matrix equation $XA=BY$, we will assume that $B$ is the generating matrix of a Reed-Solomon code, and that this matrix is also known to us.
  \item If $A$ is equivalent to $B$ up to a permutation, i.e., $A=SBP$, then $A^T=P^TB^TS^T$, and so $$AA^T=(SB)(SB)^T.$$ This factorization resembles a lot to the Cholesky decomposition of the symmetric matrix $AA^T$.
\end{itemize}

The structure of the paper is the following. First in Section \ref{permutations_vectors} we briefly discuss about two more or less standard methods to detect and determine permutations in vectors, and we focus on a third method that it is more complete than the previous two, and it has an algebraic flavour to it. Next, in Section \ref{permutations_matrices} we extend this preferred method to detecting and determining permutations in matrices. Along the way, for our personal preference, we end up computing some affine varieties. In Section \ref{attacks} we focus our attention on some cryptanalysis of the McEliece cryptosystems built on Reed-Solomon codes. It turns out that finding the corresponding affine varieties we mentioned above, is computationally very expensive. Therefore we develop an attack that uses projective codewords of minimum weight of Maximum Distance Separable (MDS) codes (see Section \ref{McElieceReedSolomon}). 

Since lots of concepts are coming into play from various areas, we tried to be as self-contained as possible, but without overloading the notes with too much information.

\medskip

\noindent{\bf Acknowledgement} We thank Dr. Alex Suchar (from University of Idaho) for the discussions on the subject, especially on the sorting method to detect permutations in vectors (Section \ref{Method_2}).

\section{Solving matrix equations $XA=BY$}\label{matrix_eqn}

\subsection{Detecting permutations in vectors.} \label{permutations_vectors}

\medskip

The group $\mathbb S_n$ {\em acts} on the vector space $\mathbb K^n$ in the following way. If $\sigma\in\mathbb S_n$ and ${\bf v}=(v_1,v_2,\ldots,v_n)\in\mathbb K^n$, then $$\sigma * {\bf v}:=(v_{\sigma(1)},v_{\sigma(2)},\ldots,v_{\sigma(n)})={\bf v}\cdot I_n(\sigma^{-1})\in\mathbb K^n.$$ In order to have a well-defined action we must have $$\sigma * (\tau * {\bf v})=(\tau\sigma) * {\bf v}.$$

Given two vectors ${\bf v},{\bf w}\in\mathbb K^n$, our goals in this subsection is to find methods to decide if there exists $\sigma\in \mathbb S_n$ such that ${\bf w}=\sigma * {\bf v}$, and determine the permutation $\sigma$.

\subsubsection{Method 1: brute-force scanning.} \label{Method_1} The first natural method to answer our goals requires to scan each entry in ${\bf w}$ and compare it to all entries in ${\bf v}$. For this technique we have at most $\displaystyle \frac{n(n-1)}{2}$ comparisons, and also we have to create a vector where we record the permutation $\sigma$:
\begin{enumerate}
  \item we take the first entry of ${\bf w}$, namely $w_1$, and compare it to all entries of ${\bf v}$ until we find the first match $w_1=v_{i_1}$;
  \item then we take the second entry of ${\bf w}$, namely $w_2$, and compare it to all entries of ${\bf v}$ except for the $i_1-$th entry, until we find the first match $w_2=v_{i_2}$;
  \item and so forth.
\end{enumerate}
If we get matchings all the way through, then $\sigma$ is the permutation $\sigma(1)=i_1,\sigma(2)=i_2,\ldots$.

\begin{exm}\label{method1} Suppose $\mathbb K=\mathbb F_7$, the prime field with 7 elements, and suppose $${\bf v}=(6,1,3,1,0,0)\mbox{ and }{\bf w}=(0,3,6,1,1,0).$$
\begin{itemize}
  \item Scan $w_1=0$, and observe that the first match is with $v_5$. Record $\sigma(1)=5$.
  \item Scan $w_2=3$ (if needed skip 5th entry in ${\bf v}$), and observe that the first match is with $v_3$. Record $\sigma(2)=3$.
  \item Scan $w_3=6$ (if needed skip 3rd and 5th entries in ${\bf v}$), and observe that the first match is with $v_1$. Record $\sigma(3)=1$.
  \item Scan $w_4=1$ (if needed skip 1st, 3rd and 5th entries in ${\bf v}$), and observe that the first match is with $v_2$. Record $\sigma(4)=2$.
  \item Scan $w_5=1$ (if needed skip 2nd, 1st, 3rd and 5th entries in ${\bf v}$), and observe that the first match is with $v_4$. Record $\sigma(5)=4$.
  \item Scan $w_6=0$, and observe that it matches with the remaining entry from ${\bf v}$, not considered yet, $v_6$. Record $\sigma(6)=6$.
\end{itemize}
So we obtained the permutation $\sigma=\left(\begin{array}{cccccc}1&2&3&4&5&6\\5&3&1&2&4&6\end{array}\right)$. Indeed $\sigma * {\bf v}={\bf w}$.

There are $2!\cdot 2!$ such permutations, because ${\bf v}$ (and also ${\bf w}$) has two entries that each repeats twice. Above we determined just one of them. Is there a nice scanning algorithm that will determine all of these permutations?
\end{exm}

\subsubsection{Method 2: total-ordering of the field.} \label{Method_2} On any set we can \underline{choose} a total-ordering of its elements (meaning that any two elements can be compared, and declare which one is ``bigger'' than the other).

For a finite field $\mathbb K$ with primitive element $\alpha$, this choosing can be done in a more standard way: we make $0$ to be the ``smallest'' element, and any $x\in\mathbb K\setminus\{0\}$ has a unique representation $x=\alpha^{i_x}$, for some unique $i_x\in\{0,\ldots,|\mathbb K|-2\}$. Then for any $x,y\in\mathbb K\setminus\{0\}$, we say that $x$ is ``smaller'' than $y$ if and only if $i_x\leq i_y$.

More particularly, if $\mathbb K=\mathbb F_p=\{0,1,\ldots,p-1\}$, then instead of using the standard order above, we could use the more natural total-ordering $0<1<2<\cdots<p-1$.

Once we have such a total-ordering, for any vector ${\bf v}\in\mathbb K^n$ there exists a permutation $\sigma_{\bf v}\in\mathbb S_n$ such that the entries of $\sigma_{\bf v}*{\bf v}$ are in increasing order. Of course, ${\bf w}=\sigma * {\bf v}$ if and only if $\sigma_{\bf v}*{\bf v}=\sigma_{\bf w}*{\bf w}$. In this instance, we have $${\bf w}=\sigma_{\bf w}^{-1}*(\sigma_{\bf v}*{\bf v})=(\sigma_{\bf v}\sigma_{\bf w}^{-1})*{\bf v},$$ so the permutation $\sigma$ is just $\sigma_{\bf v}\sigma_{\bf w}^{-1}$.

\begin{exm}\label{metod2} Using the same vectors from Example \ref{method1}, we have
$$\underbrace{(15326)}_{\sigma_{\bf v}}*{\bf v}=(0,0,1,1,3,6),$$ and $$\underbrace{(26345)}_{\sigma_{\bf w}}*{\bf w}=(0,0,1,1,3,6).$$ Note that we wrote the permutations with their disjoint cycles decomposition.

Then $\sigma=(15326)(54362)=(15423)$, which, is exactly the same permutation we obtained in Example \ref{method1}. It is not quite by chance that this happened; the way we obtained the permutations $\sigma_{\bf v}$ and $\sigma_{\bf w}$ was by ``first match'' scanning we are doing in Section \ref{Method_1}.
\end{exm}

\subsubsection{Method 3: an algebraic approach.} \label{Method_3} Let ${\bf v}=(v_1,\ldots,v_n)\in \mathbb K^n$. Define the polynomial $$Q_{\bf v}(T)=(T-v_1)(T-v_2)\cdots(T-v_n)\in\mathbb K[T],$$ and denote with $q_i({\bf v}), i=0,\ldots,n$ the coefficient of $T^i$. By Vieta's formulas, with $q_n({\bf v})=1$, for $k=1,\ldots,n$, one has $$q_{n-k}({\bf v})=(-1)^k\sum_{1\leq i_1<\cdots<i_k\leq n}v_{i_1}v_{i_2}\cdots v_{i_k}.$$ Of course, if $wt({\bf v})=m$, then $q_{n-m-1}({\bf v})=\cdots=q_0({\bf v})=0$ and $q_{n-m}({\bf v})\neq 0$.

We have the following immediate lemma:
\begin{lem} \label{equivalences} Given two vectors ${\bf v}, {\bf w}\in\mathbb K^n$, then we have the following immediate equivalent statements:
\begin{enumerate}
  \item There exists $\sigma\in\mathbb S_n$ such that $\sigma *{\bf v}={\bf w}$.
  \item $Q_{\bf v}(T)=Q_{\bf w}(T)$.
  \item For all $i=0,\ldots,n$, one has $q_i({\bf v})=q_i({\bf w})$.
  \item For all $j=1,\ldots,n$, one has $Q_{\bf v}(w_j)=0$.
\end{enumerate}
\end{lem}
This criterion seems to be more convenient to decide if a desired permutation $\sigma$ exists, but it does not determines it. Below we present a method that finds all the permutations $\sigma$, if they exist.

Let ${\bf v}=(v_1,\ldots,v_n)\in\mathbb K^n$. For each $i=1,\ldots,n$, consider the polynomial $$p_{i,{\bf v}}(T)=\frac{Q_{\bf v}(T)}{T-v_i}\in\mathbb K[T].$$ Let $D_{\bf v}(T)$ be the greatest common divisor of the polynomials $p_{1,{\bf v}}(T),\ldots,p_{n,{\bf v}}(T)$, and consider the polynomial vector:
$${\bf R}_{\bf v}(T):=\left(\frac{p_{1,{\bf v}}(T)}{D_{\bf v}(T)},\ldots,\frac{p_{n,{\bf v}}(T)}{D_{\bf v}(T)}\right)\in\mathbb K[T]^n.$$

Suppose ${\bf w}=(w_1,\ldots,w_n)\in\mathbb K^n$ is obtained by permuting the entries of ${\bf v}$. Suppose $$w_{j_1}=\cdots=w_{j_s}=v_{i_1}=\cdots=v_{i_s}.$$ Then, for all $u=1,\ldots,s$,

$${\bf R}_{\bf v}(w_{j_u})=(0,\ldots,0,\underbrace{a_u}_{i_1},0,\ldots,0,\underbrace{a_u}_{i_s},0,\ldots,0),$$ where $a_u\neq 0$. Next, consider the vector $$\mathcal R_{j_u,\bf v}=\frac{1}{a_u}{\bf R}_{\bf v}(w_{j_u}).$$

By placing in order the vectors $\mathcal R_{1,\bf v},\mathcal R_{2,\bf v},\ldots,\mathcal R_{n,\bf v}$ as the rows of a matrix, we obtain an $n\times n$ matrix with entries only 0's and 1's. In the instance of repeated entries as above, the rows $j_1,\ldots,j_s$ will be the same having 1's in positions $i_1,\ldots,i_s$, and 0's everywhere else. This is saying that we can choose $\sigma$ such that $\sigma(j_u)=i_v$, for any $u,v\in\{1,\ldots,s\}$. So from this $n\times n$ matrix we can extract an $n\times n$ ``submatrix'' that has exactly one 1 in each row and each column, hence it is a permutation matrix; namely it is $I_n(\sigma)$. If the vector ${\bf v}$ has exactly $t$ distinct entries, each repeating $s_i$ times, then, as it should happen, we have $s_1!\cdot s_2!\cdot\cdots\cdot s_t!$ options to select this submatrix.

\begin{exm}\label{method3} We apply the above method for the vectors ${\bf v}$ and ${\bf w}$ from Example \ref{method1}. We have
\begin{eqnarray}
p_{1,{\bf v}}(T)&=&(T-1)^2(T-3)T^2\nonumber\\
p_{2,{\bf v}}(T)&=&(T-6)(T-3)(T-1)T^2\nonumber\\
p_{3,{\bf v}}(T)&=&(T-6)(T-1)^2T^2\nonumber\\
p_{4,{\bf v}}(T)&=&(T-6)(T-3)(T-1)T^2\nonumber\\
p_{5,{\bf v}}(T)&=&(T-6)(T-3)(T-1)^2T\nonumber\\
p_{6,{\bf v}}(T)&=&(T-6)(T-3)(T-1)^2T.\nonumber
\end{eqnarray}

Then $D_{\bf v}(T)=(T-1)T$, and hence
\[ {\bf R}_{\bf v}(T)=\text{\tiny{\(((T-1)(T-3)T,(T-6)(T-3)T,(T-6)(T-1)T,(T-6)(T-3)T,(T-6)(T-3)(T-1),(T-6)(T-3)(T-1)).\)} } \]

\begin{eqnarray}
{\bf R}_{\bf v}(w_1)={\bf R}_{\bf v}(0)&=&(0,0,0,0,3,3)\nonumber\\
{\bf R}_{\bf v}(w_2)={\bf R}_{\bf v}(3)&=&(0,0,3,0,0,0)\nonumber\\
{\bf R}_{\bf v}(w_3)={\bf R}_{\bf v}(6)&=&(6,0,0,0,0,0)\nonumber\\
{\bf R}_{\bf v}(w_4)={\bf R}_{\bf v}(1)&=&(0,3,0,3,0,0)\nonumber\\
{\bf R}_{\bf v}(w_5)={\bf R}_{\bf v}(1)&=&(0,3,0,3,0,0)\nonumber\\
{\bf R}_{\bf v}(w_6)={\bf R}_{\bf v}(0)&=&(0,0,0,0,3,3).\nonumber
\end{eqnarray}

The $6\times 6$ matrix with entries only 0's and 1's is $\left(\begin{array}{cccccc}0&0&0&0&1&1\\0&0&1&0&0&0\\1&0&0&0&0&0\\0&1&0&1&0&0\\0&1&0&1&0&0 \\0&0&0&0&1&1 \end{array}\right)$, from which we can select the submatrix
$$I_6(\sigma)=\left(\begin{array}{cccccc}0&0&0&0&1&0\\0&0&1&0&0&0\\1&0&0&0&0&0\\0&1&0&0&0&0\\0&0&0&1&0&0 \\0&0&0&0&0&1 \end{array}\right),$$ meaning that the permutation is $\sigma=\left(\begin{array}{cccccc}1&2&3&4&5&6\\5&3&1&2&4&6\end{array}\right)=(15423)$.

The other three submatrices are the following (with the corresponding permutations):

$$\left(\begin{array}{cccccc}0&0&0&0&1&0\\0&0&1&0&0&0\\1&0&0&0&0&0\\0&0&0&1&0&0\\0&1&0&0&0&0 \\0&0&0&0&0&1 \end{array}\right)\mbox{ with permutation }(1523),$$
$$\left(\begin{array}{cccccc}0&0&0&0&0&1\\0&0&1&0&0&0\\1&0&0&0&0&0\\0&1&0&0&0&0\\0&0&0&1&0&0 \\0&0&0&0&1&0 \end{array}\right)\mbox{ with permutation }(165423),$$ and
$$\left(\begin{array}{cccccc}0&0&0&0&0&1\\0&0&1&0&0&0\\1&0&0&0&0&0\\0&0&0&1&0&0\\0&1&0&0&0&0 \\0&0&0&0&1&0 \end{array}\right)\mbox{ with permutation }(16523).$$
\end{exm}

\subsection{Detecting permutations in matrices.} \label{permutations_matrices} In this part we apply Method 3 above (Section \ref{Method_3}) to tackle the problem of solving the matrix equation $$XA=BY,$$ where $A$ and $B$ are given $k\times n$ matrices, and $X$ is an invertible $k\times k$ matrix and $Y$ is an $n\times n$ permutation matrix, both of which we want to find.

Suppose the the first row of $X$ is the vector (of indeterminates) ${\bf x}:=(x_1,x_2,\ldots,x_k)$. If $\ell_1({\bf x}),\ldots,\ell_n({\bf x})\in\mathbb K[x_1,\ldots,x_k]$ are the linear forms dual to the columns of $A$ (in this order), then first we would like to find $x_1,\ldots,x_k$ such that the vector $${\bf w}({\bf x}):={\bf x}\cdot A=(\ell_1({\bf x}),\ldots,\ell_n({\bf x})),$$ equals $\sigma * {\bf v}$, for some $\sigma\in\mathbb S_n$, where ${\bf v}$ is the first row of the given matrix $B$.

By Lemma \ref{equivalences} (4), the solutions of our problem will be points in $\mathbb K^k$, common zero locus of the polynomials
$$Q_{\bf v}(\ell_1({\bf x})), Q_{\bf v}(\ell_2({\bf x})),\ldots,Q_{\bf v}(\ell_n({\bf x})),$$ hence an {\em affine variety} (see \cite{clo2}). We denote this variety $\mathcal V_{\bf v}$.

For each solution ${\bf x}$ obtained, we will use Method 3 to determine the corresponding permutation(s) $\sigma\in \mathbb S_n$, such that ${\bf w}({\bf x})=\sigma * {\bf v}$. Denote this set of permutations with $\Lambda_{{\bf x},{\bf v}}$, and let $$\Lambda_{\bf v}:=\bigcup_{\bf x} \Lambda_{{\bf x},{\bf v}}.$$

\medskip

Next, we apply the same computations when ${\bf v}$ is the second row, third row, all the way to the last row of the matrix $B$. This way we obtain $k$ sets of permutations $\Lambda_{{\bf v}_1},\ldots,\Lambda_{{\bf v}_k}$. Since we assume that our matrix equation has a solution, the intersection $\displaystyle \Lambda:=\Lambda_{{\bf v}_1}\cap\cdots\cap\Lambda_{{\bf v}_k}$ is non-empty. Any $\sigma\in\Lambda$ will give us the matrix $Y=I_n(\sigma^{-1})$, and any $k$ vectors ${\bf x}_1,\ldots,{\bf x}_k$ with ${\bf w}({\bf x}_i)=\sigma * {\bf v}_i, i=1,\ldots,k$, will be the rows, in this order, of the matrix $X$.

Another idea to find $\sigma\in\Lambda$ without applying Method 3 $k$ times, but just once, is to check which permutations from $\Lambda_{{\bf v}_1}$ are in $\Lambda_{{\bf v}_2}$, and from this smaller set to check which permutation is in $\Lambda_{{\bf v}_3}$, and so forth.

\begin{exm}\label{XA=BY} Suppose $\mathbb K=\mathbb R$, and let
$$A=\left(\begin{array}{rrrr}8&-1&1&2\\11&-2&1&3 \end{array}\right)\mbox{ and }B=\left(\begin{array}{rrrr}1&0&1&2\\-1&1&0&3 \end{array}\right).$$

We have ${\bf x}=(x_1,x_2)$, and $$\ell_1({\bf x})=8x_1+11x_2,\ell_2({\bf x})=-x_1-2x_2,\ell_3({\bf x})=x_1+x_2,\ell_4({\bf x})=2x_1+3x_2.$$

$\bullet$ Solving the system (using \cite{GrSt})
\begin{eqnarray}
(\ell_1({\bf x})-1)(\ell_1({\bf x})-0)(\ell_1({\bf x})-1)(\ell_1({\bf x})-2)&=&0\nonumber\\
(\ell_2({\bf x})-1)(\ell_2({\bf x})-0)(\ell_2({\bf x})-1)(\ell_2({\bf x})-2)&=&0\nonumber\\
(\ell_3({\bf x})-1)(\ell_3({\bf x})-0)(\ell_3({\bf x})-1)(\ell_3({\bf x})-2)&=&0\nonumber\\
(\ell_4({\bf x})-1)(\ell_4({\bf x})-0)(\ell_4({\bf x})-1)(\ell_4({\bf x})-2)&=&0\nonumber
\end{eqnarray} gives the solutions $\{(3,-2),(0,0)\}$. We drop the solution $(0,0)$ because this vector cannot be the first row of the invertible matrix $X$.

We then have ${\bf w}(3,-2)=(2,1,1,0)$, and $$\Lambda_{{\bf v}_1}=\{(142),(1423)\}.$$

$\bullet$ Solving the system (using \cite{GrSt})
\begin{eqnarray}
(\ell_1({\bf x})+1)(\ell_1({\bf x})-1)(\ell_1({\bf x})-0)(\ell_1({\bf x})-3)&=&0\nonumber\\
(\ell_2({\bf x})+1)(\ell_2({\bf x})-1)(\ell_2({\bf x})-0)(\ell_2({\bf x})-3)&=&0\nonumber\\
(\ell_3({\bf x})+1)(\ell_3({\bf x})-1)(\ell_3({\bf x})-0)(\ell_3({\bf x})-3)&=&0\nonumber\\
(\ell_4({\bf x})+1)(\ell_4({\bf x})-1)(\ell_4({\bf x})-0)(\ell_4({\bf x})-3)&=&0\nonumber
\end{eqnarray} gives the solutions $\{(-1,1),(0,0)\}$. We drop the solution $(0,0)$ because this vector cannot be the second row of the invertible matrix $X$.

We then have ${\bf w}(-1,1)=(3,-1,0,1)$, and $$\Lambda_{{\bf v}_2}=\{(142)\}.$$

\medskip

We obtained $\displaystyle\Lambda=\Lambda_{{\bf v}_1}\cap \Lambda_{{\bf v}_2}=\{(142)\}$, hence the solution of the matrix equation $XA=BY$ is
$$Y=I_4((241))=\left(\begin{array}{cccc}0&1&0&0\\0&0&0&1\\0&0&1&0\\1&0&0&0 \end{array}\right)\mbox{ and }X=\left(\begin{array}{rr}3&-2\\-1&1 \end{array}\right).$$
\end{exm}

Unfortunately, finding the affine variety $\mathcal V_{\bf v}$, requires a lot of computer power. If the polynomial $Q_{\bf v}(T)$ has $u$ distinct factors, then in order to compute this variety we need to solve $u^n$ linear systems of $n$ equations and $k$ indeterminates. Even for simple examples, such as Example \ref{attack_mincodewords} below, Macaulay 2 (\cite{GrSt}) failed to complete the task in reasonable time.
\section{Application: an attack on McEliece cryptosystems built on Reed-Solomon codes}\label{attacks}

\subsection{Linear Codes.} \label{linear_codes} The basics of coding theory can be found for example in \cite{hp1}, or for a very friendly introduction, in \cite{TrWa}. Below we sum up some of the most important concepts and techniques.

Let $\mathcal C$ be an $[n,k,d]-$linear code with generating matrix (in canonical bases) $$G=\left(\begin{array}{cccc}a_{11}&a_{12}&\cdots&a_{1n}\\ a_{21}&a_{22}&\cdots&a_{2n}\\\vdots&\vdots& &\vdots\\ a_{k1}&a_{k2}&\cdots&a_{kn}\end{array}\right),$$ where $a_{ij}\in\mathbb K$, any field.

By this, one understands that $\mathcal C$ is the image of the injective linear map $$\phi:\mathbb K^k\stackrel{G}\longrightarrow \mathbb K^n.$$ $n$ is the {\em length} of $\mathcal C$, $k$ is the {\em dimension} of $\mathcal C$ and $d$ is the {\em minimum distance (or Hamming distance)}, the smallest number of non-zero entries in a non-zero element of $\mathcal C$.

The elements of $\mathcal C$ are called {\em codewords}. To decode a codeword ${\bf y}$, means to find (the unique) ${\bf x}\in\mathbb K^k$ such that ${\bf x}\cdot G={\bf y}$.

\subsubsection{Decoding via Gaussian elimination.} \label{decoding} One rather simple technique to find ${\bf x}=(x_1,\ldots,x_k)$ is the following: to the generating matrix $G$ append ${\bf y}$ as a last row, and the vector of new indeterminates $(r_1,\ldots,r_k,y)^T$ as a last column, to obtain a new matrix $\bar{G}_{\bf y}$ of size $(k+1)\times(n+1)$. If one does Gaussian elimination to bring $\bar{G}_{\bf y}$ to echelon form, the entry in position $(k+1,n+1)$ is exactly $y-(x_1r_1+\cdots+x_kr_k)$ that should equal 0. Hence we obtained the vector ${\bf x}$. This works because ${\bf y}$ is the linear combination of the $k$ rows of $G$, with coefficients $x_1,\ldots,x_k$.

As an example, with $\mathbb K=\mathbb F_3$, suppose $G=\left(\begin{array}{ccccc}1&0&2&1&0\\1&1&0&1&2\\0&2&1&2&1\end{array}\right)$ and ${\bf y}=(2,2,1,0,1)$. The new matrix is $\bar{G}_{\bf y}=\left(\begin{array}{ccccccc}1&0&2&1&0&|&r_1\\1&1&0&1&2&|&r_2\\0&2&1&2&1&|&r_3\\1&1&2&0&2&|&y\end{array}\right),$ with echelon form
$$\left(\begin{array}{ccccccc}1&0&2&1&0&|&r_1\\0&1&1&0&2&|&2r_1+r_2\\0&0&2&2&0&|&2r_1+r_2+3r_3\\0&0&0&0&0&|&y+2r_1+2r_2+r_3\end{array}\right).$$ This implies that $y+2r_1+2r_2+r_3=0$, which gives ${\bf x}=(1,1,2)$.

\medskip

In Section \ref{Introduction} we mentioned already the parity check matrix of a code. This $(n-k)\times n$ matrix $H:=[-A^T|I_{n-k}]$ is important because the fundamental property: ${\bf y}=(y_1,\ldots,y_n)\in\mathcal C$ if and only if $H{\bf y}^T=0$. So the entries of ${\bf y}$ satisfy $n-k$ homogeneous linear equations; in fact, these $n-k$ relations form a basis for the space, denoted $F(\mathcal C)$, of the linear dependencies among the columns of $G$.

\medskip

For any vector ${\bf w}\in\mathbb K^n$, the {\em weight} of ${\bf w}$, denoted $wt({\bf w})$, is the number of non-zero entries in ${\bf w}$. Codewords of weight equal to $d$ are called {\em codewords of minimum weight}, so such a codeword is a non-zero vector ${\bf y}=(y_1,\ldots,y_n)$ in the kernel of the parity-check matrix, for which there exist $1\leq j_1<\ldots<j_{n-d}\leq n$ with $y_{j_1}=\cdots=y_{j_{n-d}}=0$. It is clear that any such vector multiplied by a non-zero scalar will lead to another codeword of minimum weight. Therefore, it is more convenient to consider {\em projective} codewords of minimum weight, which will mean the equivalence classes of the equivalence relation: ${\bf w}_1\sim{\bf w}_2$ if and only if ${\bf w}_1=c{\bf w}_2,$ for some $c\in\mathbb\setminus\{0\}$. The equivalence class of ${\bf w}$ is denoted $[{\bf w}]$.

To sum up, the set of projective codewords of minimum weight of an $[n,k,d]-$linear code $\mathcal C$ form the projective variety $$\mathcal V(\mathcal C,d+1):=V(F(\mathcal C),\{y_{i_1}\cdots y_{i_{d+1}}|\mbox{ for all }1\leq i_1<\cdots<i_{d+1}\leq n\})\subset \mathbb P^{n-1};$$ see \cite[Section 3.2]{GaTo}) for details. This remark is the analog of the method described in \cite{t} which was based on \cite{dp}, but using projective varieties in $\mathbb P^{n-1}$, instead of projective varieties in $\mathbb P^{k-1}$.

\subsubsection{Reed-Solomon codes.} \label{Reed-Solomon} We present these classical linear codes from the point of view of BCH codes (see \cite[Section 18.9]{TrWa}). Let $\mathbb K:=\mathbb F_q$ be the finite field with $q$ elements. Let $n=q-1$, and let $\alpha$ be a primitive $n-$th root of unity. Let $1\leq d<n$ and consider the polynomial $$g(x):=(x-\alpha)(x-\alpha^2)\cdots(x-\alpha^{d-1})=g_0+g_1x+\cdots+g_{d-2}x^{d-2}+x^{d-1}\in \mathbb F_q[x].$$ The {\em Reed-Solomon code} is defined to be the linear code with generating matrix of size $k\times n$, with $k=n+1-d$,

$$G=\left(\begin{array}{cccccccc}g_0&g_1&\cdots&g_{d-2}&1&0&\cdots&0\\0&g_0&g_1&\cdots&g_{d-2}&1&\cdots&0\\ \vdots&&\ddots&\ddots&&\ddots&\ddots&\\ 0&\cdots&0&g_0&g_1&\cdots&g_{d-2}&1\end{array}\right), g_i\neq 0, i=0,\ldots, d-2.$$ It is known that the Reed-Solomon code has parameters $[n,k,d]$, hence, because $d=n-k+1$, it is a {\em Maximum Distance Separable (MDS)} code.

If $\mathcal C$ is an $[n,k,d]-$linear code with $d=n-k+1$ (so an MDS code), then the number of projective codewords of minimum weight is $\displaystyle {{n}\choose{d}}$ (see for example \cite[Corollary 3.3]{t}). In the setup of MDS codes, the question we want to answer is the following: for any choice of $1\leq i_1<\cdots<i_{n-d}\leq n$, is $V(F(\mathcal C),y_{i_1},\ldots,y_{i_{n-d}})$ a single projective codeword of minimum weight? By the lemma below, the answer is yes.

\begin{lem} Let $\mathcal C$ be an $[n,k,d]$ MDS code. Then, for any $1\leq i_1<\cdots<i_{n-d}\leq n$, one has
$$\dim_{\mathbb K}{\rm Span}_{\mathbb K}(F(\mathcal C)\cup\{y_{i_1},\ldots,y_{i_{n-d}}\})=n-1.$$
\end{lem}
\begin{proof} Let $M$ be a generating matrix of $\mathcal C$, and suppose there exist an element of $F(\mathcal C)$ of the form $c_1y_{i_1}+\cdots+c_{n-d}y_{i_{n-d}}$. Then, the corresponding $n-d=k-1$ columns of $M$ are linearly dependent, so they span a subspace of dimension $<k-1$. By \cite[Remarks 2.2 and 2.3]{ToVa} we obtain that the minimum distance of $\mathcal C$ is $<n-k+1$, which is a contradiction.

So the $n-k$ elements of the basis of $F(\mathcal C)$ together with $y_{i_1},\ldots,y_{i_{n-d}}$, are $n-k+n-d=n-1$ linearly independent linear forms.
\end{proof}

\subsection{McEliece cryptosystems.} \label{McEliece} We follow the description of this public key cryptosystem according to \cite[Section 18.10]{TrWa}.

{\em Bob} chooses a $k\times n$ matrix $G$ that is the generating matrix of an $[n,k,d]-$linear code, denoted $\mathcal C$. Also he chooses a $k\times k$ invertible matrix $S$, and a $n\times n$ permutation matrix $P$. {\em Bob} keeps the matrices $G, S, P$ secret and makes public (the public key) the $k\times n$ matrix $$G_1:=SGP.$$

{\em Alice} wants to send privately to {\em Bob} the plain text ${\bf x}\in\mathbb K^{k}$. To encrypt ${\bf x}$, {\em Alice} chooses randomly a vector ${\bf e}\in\mathbb K^n$ with $wt({\bf e})=t\leq\lfloor(d-1)/2\rfloor$. She forms the the ciphertext by computing $${\bf y}:={\bf x}G_1+{\bf e}.$$

{\em Bob} decrypts ${\bf y}$ as follows:
\begin{enumerate}
  \item He calculates ${\bf y}_1={\bf y}P^{-1}={\bf x}SG+{\bf e}_1$, where ${\bf e}_1={\bf e}P^{-1}$. Since $P^{-1}$ is also a permutation matrix, the weight of ${\bf e}_1$ also equals $t$.
  \item Since $S$ is invertible, and since $SG$ is a $k\times n$ matrix with its rows being linear combinations of the rows of $G$, then $SG$ is a generating matrix for the same linear code $\mathcal C$. {\em Bob} ``error-corrects'' ${\bf y}_1$ to get rid of the ``error'' ${\bf e}_1$, and obtains the nearest-neighbor codeword $\bar{\bf y}:={\bf x}SG\in\mathcal C$. The condition $t\leq\lfloor(d-1)/2\rfloor$ ensures that $\bar{\bf y}$ is unique.
  \item {\em Bob} decodes $\bar{\bf y}$: he finds $\bar{\bf x}\in\mathbb K^k$ with $\bar{\bf y}=\bar{\bf x}G$.
  \item {\em Bob} computes ${\bf x}=\bar{\bf x}S^{-1}$.
\end{enumerate}

The secret matrix $G$ is chosen in such manner that it is very efficient to encrypt ${\bf x}$ and decrypt ${\bf y}$, based upon effective algorithms to error-correct and decode this linear code (choosing Goppa codes seems to balance the security requirements and the efficiency mentioned above). An outside attacker has knowledge only of the public key, which is the matrix $G_1$. So a priori, in order to error-correct and decode the intercepted ciphertext ${\bf y}$ to find the secret plaintext ${\bf x}$, the attacker needs to apply ``general'' algorithms; and this is known to be an NP-hard problem. But if the permutation matrix $P$ is known to the attacker, then the matrix $G_1P^{-1}$ is a generating matrix of the \underline{same} secret chosen code. Knowing what kind of code was chosen initially, the attacker can use the same efficient algorithms to error-correct and decode. So keeping the permutation matrix secret is of utmost importance for the security of the McEliece cryptosystem.

\subsubsection{McEliece cryptosystems built on Reed-Solomon codes.} \label{McElieceReedSolomon} In what follows we assume that $G$ is the generating matrix of a Reed-Solomon code.

If the outside attacker knows what is the chosen primitive element $\alpha$ of the base field $\mathbb K:=\mathbb F_q$, then he/she will know what is the matrix $G$ (see Section \ref{Reed-Solomon}). This is because $G_1$ is a $k\times n$ generating matrix of an MDS $[n,k]-$linear code, hence $d=n-k+1$.

Suppose the outside attacker doesn't know the chosen primitive element of $\mathbb K$. Since $\mathbb K$ is the finite field with $q$ elements, then $q=p^z$, for some prime number $p$, and so the number of primitive elements of $\alpha\in\mathbb F_q$ equals the Euler's function evaluated at $q-1$, i.e., $\phi(q-1)$. In this instance, for each primitive element $\alpha$, the matrix $G_{\alpha}$ that generates the corresponding Reed-Solomon code is known. So the attacker may try all the possible primitive elements and see which one works; obviously not the best strategy to pursue. It seems that finding the right primitive element translates into a complex Discrete Log problem, and this will be the focus of a future project (also check Section \ref{primitive}).

\medskip

Suppose we are the attackers, and we know the chosen primitive element $\alpha$. The our goal is to find the matrices $S$ and $P$ from $G_1=SGP$. As we mentioned at the beginning, we want to solve the matrix equation $$XA=BY,$$ when $A=G_1$, and $B=G$ are given.

\medskip

Denote ${\bf g}:=(g_0,\ldots,g_{d-2},1,0,\ldots,0)$, the first row of $G$. Let $\tau\in\mathbb S_n$ denote the cycle $(n\, n-1\,\cdots\, 2\, 1)$. Then the second row of $G$ is $\tau*{\bf g}$, the third row of $G$ is $\tau^2*{\bf g}$, and so forth. This observation reduces drastically the computations. Instead of solving $k$ systems of equations
\[
\left\{
  \begin{array}{ccc}
    Q_{\bf v}(\ell_1({\bf x}))&=&0 \\
    Q_{\bf v}(\ell_2({\bf x}))&=&0 \\
    &\vdots& \\
    Q_{\bf v}(\ell_n({\bf x}))&=&0,
  \end{array}
\right.
\] with ${\bf v}$ scanning over all rows of $G$, it is enough to solve just one system when ${\bf v}={\bf g}$.

The solution will consist of at least $k$ points ${\bf x}_1,\ldots,{\bf x}_m$, and for each $i=1,\ldots,m$ define
$$\Delta_i:=\{\sigma\in\mathbb S_n|{\bf w}({\bf x}_i)=\sigma*{\bf g}\}.$$

What remains to be found is a set of distinct indices $i_1,\ldots,i_k\in\{1,\ldots,m\}$, and a permutation $\sigma$ such that $$\tau^{j-1}\sigma\in \Delta_{i_j},\mbox{ for all }j=1,\ldots,k.$$ In this instance, $\sigma$ will give the permutation matrix $P=I_n(\sigma^{-1})$, and ${\bf x}_{i_1},\ldots,{\bf x}_{i_k}$ will give, in this particular order, the rows of the matrix $S^{-1}$.

\medskip

We have that for all $j=1,\ldots,k$, $${\bf w}({\bf x}_{i_j})=(\tau^{j-1}\sigma)*{\bf g},$$ for some $\sigma\in\mathbb S_n$.

Since ${\bf w}({\bf x}_{i_1})=(\sigma)*{\bf g}$, we can determine all other ${\bf x}_{i_2},\ldots,{\bf x}_{i_k}$, via the formula
$${\bf w}({\bf x}_{i_j})=(\sigma^{-1}\tau^{j-1}\sigma)*{\bf w}({\bf x}_{i_1})={\bf w}({\bf x}_{i_1})\cdot I_n(\sigma^{-1}(\tau^{-1})^{j-1}\sigma), j=2,\ldots,k.$$

To sum up, the way this attack is going is the following:
\begin{itemize}
  \item[STEP 1:] Compute the affine variety $\mathcal V_{\bf g}:=V(Q_{\bf g}(\ell_1({\bf x})),\ldots,Q_{\bf g}(\ell_n({\bf x})))=\{{\bf x}_1,\ldots,{\bf x}_m\}$.
  \item[STEP 2:] Record all the vectors $W:=\{{\bf w}({\bf x}_1),\ldots,{\bf w}({\bf x}_m)\}$.
  \item[STEP 3:] Pick an ${\bf x}_{i_0}\in\mathcal V$, and compute $\Delta_{i_0}$.
  \item[STEP 4:] Pick a $\sigma\in\Delta_{i_0}$ and calculate $(\sigma^{-1}\tau^{j-1}\sigma)*{\bf w}({\bf x}_{i_0})$ for $j=2,\ldots,k$, where $\tau=(n\, n-1\,\cdots\, 2\, 1)$.
  \item[STEP 5:] If there is one vector at STEP 4 that does not belong to $W$, then pick a different permutation in $\Delta_{i_0}$, and repeat STEP 4.
  \item[STEP 6:] If for each permutation in $\Delta_{i_0}$ there is a vector not in $W$, then pick a different element of $\mathcal V$, and redo the algorithm from STEP 3.
\end{itemize}

We mentioned already that computing the affine variety $\mathcal V_{\bf g}$ is a challenge, so in what follows we will use codewords of minimum weight to avoid STEP 1 above.

\medskip

\noindent AN ATTACK USING CODEWORDS OF MINIMUM WEIGHT. For this attack we have to rely on projective geometry. The projective vector $$[{\bf g}]:=[g_0,\ldots,g_{d-2},1,0,\ldots,0],$$ which is the first row of $G$ is a projective codeword of minimum weight of $G_1$, after permuting its entries according to the matrix $P$. In fact each of the $k$ rows of $G$ will lead to a distinct projective codeword of minimum weight for $G_1$.

Since $G_1$ generates an MDS code (i.e., $d=n-k+1$), from the $\displaystyle {{n}\choose{d}}$ projective codewords of minimum weight of this code, we will select those that, after a permutation, equal to $[{\bf g}]$. Once these were determined, then STEPS 1 and 2 in the previous algorithm are completed, and we can continue with the remaining STEPS to determine the matrix $P$. After this, a simple Gaussian elimination algorithm will give the matrix $S$.

Let us rewrite the notations and constructions used at the beginning of Section \ref{Method_3}. Suppose ${\bf v}=(v_1,\ldots,v_n)\in \mathbb K^n$ is a vector of weight $m$. Define the polynomial $$Q_{\bf v}(T)=(T-v_1)(T-v_2)\cdots(T-v_n)\in\mathbb K[T],$$ and denote with $q_i({\bf v}), i=0,\ldots,n$ the coefficient of $T^i$. By Vieta's formulas, with $q_n({\bf v})=1$, for $k=1,\ldots,n$, one has $$q_{n-k}({\bf v})=(-1)^k\sum_{1\leq i_1<\cdots<i_k\leq n}v_{i_1}v_{i_2}\cdots v_{i_k}.$$ Of course, if $wt({\bf v})=m$, then $q_{n-m-1}({\bf v})=\cdots=q_0({\bf v})=0$ and $q_{n-m}({\bf v})\neq 0$.

\begin{lem} \label{projective} Let ${\bf v}=(v_1,\ldots,v_n),{\bf w}=(w_1,\ldots,w_n)\in\mathbb K^n$. Then, there exist $\sigma\in \mathbb S_n$ such that $[{\bf w}]=[\sigma*{\bf v}]$ if and only if there exist $c\in\mathbb K\setminus\{0\}$ such that for all $i=0,\ldots,n$, one has $q_i({\bf w})=c^{n-i}q_i({\bf v})$.
\end{lem}

\begin{proof} We have that $[{\bf w}]=[\sigma*{\bf v}]$ if and only if there exists $c\in\mathbb K\setminus\{0\}$ such that ${\bf w}=c(\sigma*{\bf v})=\sigma*(c{\bf v})$. Then, from Lemma \ref{equivalences} (3), this is equivalent to $$q_i({\bf w})=q_i(c{\bf v})=(-1)^{n-i}\sum_{1\leq i_1<\cdots<i_{n-i}\leq n}(cv_{i_1})\cdots (cv_{i_{n-i}})=c^{n-i}q_i({\bf v}).$$
\end{proof}

For convenience, if ${\bf v}$ is a vector of weight $m$, and since $q_n({\bf v})=1$, we consider the vector
$${\bf q}({\bf w})=(-q_{n-1}({\bf w}),q_{n-2}({\bf w}),\ldots,(-1)^mq_{n-m}({\bf w})).$$ Then, with these notations, we have that $[{\bf w}]=[\sigma*{\bf v}]$ and $wt({\bf w})=wt({\bf v})=d$, if and only if there exists $c\in\mathbb K\setminus\{0\}$ such that
$${\bf q}({\bf w})={\bf q}({\bf v})\cdot \left(\begin{array}{cccc}c&0&\cdots&0\\0&c^2&\cdots&0\\ \vdots&\vdots&\ddots&\vdots\\0&0&\cdots&c^d\end{array}\right).$$

\medskip

If we want to select all the codewords of minimum weight whose entries are the same as the entries of ${\bf g}=(g_0,\ldots,g_{d-2},1,0,\ldots,0)$, but in different positions, one has to do the following:

$\bullet$ First find all projective codewords of minimum weight of the linear code $\mathcal C_1$ with generating matrix $G_1$, by determining $\mathcal V(\mathcal C_1,d+1)$ (see Section \ref{linear_codes}).

$\bullet$ Use Lemma \ref{projective}, and the comments after it to select those $[{\bf w}]\in \mathcal V(\mathcal C_1,d+1)$ that are permutations of $[{\bf g}]$.

$\bullet$ To find the right representative we determine the corresponding $c$, and multiply ${\bf w}$ by $c^{-1}$.

This way we obtain the list of vectors $W$ mentioned at STEP 2 in Section \ref{McElieceReedSolomon}. To finish this attack, we proceed with the remaining steps presented back in that section, with the caveat that after we obtain the appropriate $k$ vectors in $W$ that give the desired permutation matrix $P$, we have to decode these vectors to obtain the rows of $S^{-1}$. For this purpose, we can use the decoding algorithm briefly described in Section \ref{decoding}.

\begin{exm} \label{attack_mincodewords} Suppose $\mathbb K:=\mathbb F_7$, the prime field with $7$ elements, with primitive element considered being $\alpha=3$. Then $$G=\left(\begin{array}{cccccc}6&1&3&1&0&0\\0&6&1&3&1&0\\0&0&6&1&3&1\end{array}\right),$$ so $n=6$, $d=4$, and $k=3$ (we borrowed the first Example in \cite[Chapter 18.9]{TrWa}). Suppose we have
$$G_1=\left(\begin{array}{cccccc}1&2&3&6&6&1\\4&0&4&6&1&0\\0&2&5&4&6&3\end{array}\right).$$

The parity-check matrix is
$$H_{G_1}=\left(\begin{array}{cccccc}3&4&6&1&0&0\\1&1&4&0&1&0\\3&6&4&0&0&1\end{array}\right),$$ and hence $$F(\mathcal C)=\langle 3y_1+4y_2+6y_3+y_4,y_1+y_2+4y_3+y_5,3y_1+6y_2+4y_3+y_6\rangle.$$

The projective codewords of minimum weight are the $\displaystyle {{6}\choose{4}}=15$ distinct solutions of the system of equations
\begin{eqnarray}
3y_1+4y_2+6y_3+y_4&=&0\nonumber\\
y_1+y_2+4y_3+y_5&=&0\nonumber\\
3y_1+6y_2+4y_3+y_6&=&0\nonumber\\
y_1y_2y_3y_4y_5&=&0\nonumber\\
y_1y_2y_3y_4y_6&=&0\nonumber\\
y_1y_2y_3y_5y_6&=&0\nonumber\\
y_1y_3y_4y_5y_6&=&0\nonumber\\
y_1y_2y_4y_5y_6&=&0\nonumber\\
y_2y_3y_4y_5y_6&=&0.\nonumber
\end{eqnarray}

We have ${\bf g}=(6,1,3,1,0,0)$, so $${\bf q}({\bf g})=(4,2,3,4).$$

We have the following table of results

\begin{center}
\begin{tabular}{|c|c|c|c|c|c|c|}
\hline\hline
$[{\bf w}]$ & ${\bf q}({\bf w})$& $c$ & $c^2$ & $c^3$ & $c^4$ & Select?\\
\hline\hline
$[5,5,1,1,0,0]$ & $(5,4,4,4)$ & $5\cdot 4^{-1}=3$ & $4\cdot 2^{-1}=2=3^2$ & $4\cdot 3^{-1}=6=3^3$ & $4\cdot 4^{-1}=1\neq 3^4$ & No\\
\hline
$[0,4,6,4,0,1]$ & $(1,1,6,5)$ & $1\cdot 4^{-1}=2$ & $1\cdot 2^{-1}=4=2^2$ & $6\cdot 3^{-1}=2\neq 2^3$ & & No\\
\hline
$[3,0,1,6,0,1]$ & $(4,2,3,4)$ & $4\cdot 4^{-1}=1$ & $2\cdot 2^{-1}=1=1^2$ & $3\cdot 3^{-1}=1=1^3$ &$4\cdot 4^{-1}=1=1^4$ & Yes\\
\hline
$[5,2,0,5,0,1]$ & $(6,1,4,1)$ & $6\cdot 4^{-1}=5$ & $1\cdot 2^{-1}=4=5^2$ & $4\cdot 3^{-1}=6=5^3$ &$1\cdot 4^{-1}=2=5^4$ & Yes\\
\hline
$[1,5,2,0,0,1]$ & $(2,4,6,3)$ & $2\cdot 4^{-1}=4$ & $4\cdot 2^{-1}=2=4^2$ & $6\cdot 3^{-1}=2\neq 4^3$ & & No\\
\hline
$[5,1,0,2,1,0]$ & $(2,4,6,3)$ & $2\cdot 4^{-1}=4$ & & $6\cdot 3^{-1}=2\neq 4^3$ & & No\\
\hline
$[4,0,4,6,1,0]$ & $(1,1,6,5)$ & $1\cdot 4^{-1}=2$ & & $6\cdot 3^{-1}=2\neq 2^3$ & & No\\
\hline
$[0,3,6,1,1,0]$ & $(4,2,3,4)$ & $4\cdot 4^{-1}=1$ & $2\cdot 2^{-1}=1=1^2$ & $3\cdot 3^{-1}=1=1^3$ &$4\cdot 4^{-1}=1=1^4$ & Yes\\
\hline
$[2,5,5,0,1,0]$ & $(6,1,4,1)$ & $6\cdot 4^{-1}=5$ & $1\cdot 2^{-1}=4=5^2$ & $4\cdot 3^{-1}=6=5^3$ &$1\cdot 4^{-1}=2=5^4$ & Yes\\
\hline
$[0,1,0,3,6,1]$ & $(4,2,3,4)$ & $4\cdot 4^{-1}=1$ & $2\cdot 2^{-1}=1=1^2$ & $3\cdot 3^{-1}=1=1^3$ &$4\cdot 4^{-1}=1=1^4$ & Yes\\
\hline
$[6,0,4,0,6,1]$ & $(3,2,4,4)$ & $3\cdot 4^{-1}=6$ & $2\cdot 2^{-1}=1=6^2$ & $4\cdot 3^{-1}=6=6^3$ &$4\cdot 4^{-1}=1=6^4$ & Yes\\
\hline
$[0,0,5,5,1,1]$ & $(5,4,4,4)$ & $5\cdot 4^{-1}=3$ & & & $4\cdot 4^{-1}=1\neq 3^4$ & No\\
\hline
$[3,3,0,0,1,1]$ & $(1,1,3,2)$ & $1\cdot 4^{-1}=2$ & $1\cdot 2^{-1}=4=2^2$ & $3\cdot 3^{-1}=1=2^3$ & $2\cdot 4^{-1}=4\neq 2^4$ & No\\
\hline
$[2,0,0,1,5,1]$ & $(2,4,6,3)$ & $2\cdot 4^{-1}=4$ & & $6\cdot 3^{-1}=2\neq 4^3$ & & No\\
\hline
$[0,6,3,0,3,1]$ & $(6,1,1,5)$ & $6\cdot 4^{-1}=5$ & $1\cdot 2^{-1}=4=5^2$ & $1\cdot 3^{-1}=5\neq 5^3$ & & No\\
\hline\hline
\end{tabular}
\end{center}

We got six selections. After multiplying by the corresponding $c^{-1}$, we get the six vectors which are the rows of the matrix:

$$W=\left(\begin{array}{cccccc}3&0&1&6&0&1\\1&6&0&1&0&3\\0&3&6&1&1&0\\6&1&1&0&3&0\\0&1&0&3&6&1\\1&0&3&0&1&6\end{array}\right).$$

Next we pick up and continue using the algorithm in Section \ref{McElieceReedSolomon}.

\medskip

\noindent STEP 3: Let us pick the third row of $W$ as our $${\bf w}({\bf x}_{i_0})=(0,3,6,1,1,0)=:{\bf w}.$$

\medskip

\noindent STEP 4: In Example \ref{method3}, we computed $\Delta_{i_0}$. Let's pick $\sigma=(15423)\in\Delta_{i_0}$. With $\tau=(654321)$ we calculate

$${\bf w}\cdot I_6(\sigma^{-1}\tau^{-1}\sigma)={\bf w}\cdot I_6((163425))=(1,1,0,6,3,0),$$ which is not in $W$.

\medskip

\noindent STEP 5: Let's pick $\sigma=(16523)\in\Delta_{i_0}$. Then

$${\bf w}\cdot I_6(\sigma^{-1}\tau^{-1}\sigma)={\bf w}\cdot I_6((135246))=(0,1,0,3,6,1),$$ and, with $(\tau^2)^{-1}=(135)(246)$,

$${\bf w}\cdot I_6(\sigma^{-1}(\tau^2)^{-1}\sigma)={\bf w}\cdot I_6((154)(263))=(1,6,0,1,0,3).$$ These vectors are the fifth, and respectively, the second rows of $W$.

\medskip

This is telling us that the permutation matrix $P$ is
$$P= I_6((16523)^{-1})=\left(\begin{array}{cccccc}0&0&1&0&0&0\\0&0&0&0&1&0\\0&1&0&0&0&0\\0&0&0&1&0&0\\0&0&0&0&0&1\\1&0&0&0&0&0\end{array}\right).$$

\medskip

$\bullet$ Decoding the codeword of minimum weight $(0,3,6,1,1,0)$, we get the first row of the matrix $S^{-1}$ to be $(4,6,1)$.

$\bullet$ Decoding the codeword of minimum weight $(0,1,0,3,6,1)$, we get the second row of the matrix $S^{-1}$ to be $(2,3,2)$.

$\bullet$ Decoding the codeword of minimum weight $(1,6,0,1,0,3)$, we get the third row of the matrix $S^{-1}$ to be $(3,3,0)$.

So $$S^{-1}= \left(\begin{array}{ccc}4&6&1\\2&3&2\\3&3&0\end{array}\right).$$
\end{exm}

\subsubsection{Finding the chosen primitive element.}\label{primitive} In this section we briefly discuss a method to discover the chosen primitive element $\alpha$ of $\mathbb K$, only if we have knowledge of the matrix $G_1$.

By Section \ref{Reed-Solomon} the first row of the $k\times n$ matrix $G_{\alpha}$ is the vector $${\bf g}_{\alpha}:=(g_0(\alpha),g_1(\alpha),\ldots,g_{d-2}(\alpha),\underbrace{g_{d-1}(\alpha)}_1,0,\ldots,0),$$ with $d=n-k+1$, where $g_i(\alpha)$ is the coefficient of $x^i$ in $g_{\alpha}(x):=(x-\alpha)(x-\alpha^2)\cdots(x-\alpha^{d-1})$, for $i=0,\ldots,d-1$, hence
$$g_i(\alpha)=(-1)^{(d-1)-i}\sum_{1\leq j_1<\cdots< j_{(d-1)-i}\leq d-1}\alpha^{j_1+\cdots+j_{(d-1)-i}}.$$

Withe these we construct the familiar vector $${\bf q}({\bf g}_{\alpha})=\left(\ldots,\sum_{0\leq i_1<\cdots<i_j\leq d-1}(g_{i_1}(\alpha)\cdots g_{i_j}(\alpha)),\ldots\right).$$

Suppose we determined $[{\bf w}]$, a projective codeword of minimum weight (equal to $d$). Suppose ${\bf q}({\bf w})=(A_1,\ldots,A_d)\in\mathbb K^d$. By solving the following system of equations in unknowns $X$ and $Y$, we determine if there is a scalar $c\neq 0$, and a primitive element $\alpha$ such that ${\bf w}$ is some permutation of the vector $c{\bf g}_{\alpha}$.

\begin{eqnarray}
YA_1&=&g_0(X)+\cdots+g_{d-1}(X)\nonumber\\
&\vdots&\nonumber\\
Y^jA_j&=&\sum_{0\leq i_1<\cdots<i_j\leq d-1}(g_{i_1}(X)\cdots g_{i_j}(X))\nonumber\\
&\vdots&\nonumber\\
Y^dA_d&=&g_0(X)\cdots g_{d-1}(X).\nonumber
\end{eqnarray} The solution, if it exists, will be $Y=c^{-1}$, and $X=\alpha$.

Because $g_{d-1}(X)=1$, we can rearrange the equations in the following way.

\begin{eqnarray}
g_0(X)+\cdots+g_{d-2}(X)&=&YA_1-1\nonumber\\
\sum_{0\leq i<j\leq d-2}(g_i(X)g_j(X))&=&Y^2A_2-YA_1+1\nonumber\\
&\vdots&\nonumber\\
\sum_{0\leq i_1<\cdots<i_j\leq d-2}(g_{i_1}(X)\cdots g_{i_j}(X))&=&Y^jA_j-Y^{j-1}A_{j-1}+\cdots+(-1)^{j-1}YA_1+(-1)^j\nonumber\\
&\vdots&\nonumber\\
g_0(X)\cdots g_{d-2}(X)&=&Y^{d-1}A_{d-1}-\cdots+(-1)^{d-2}YA_1+(-1)^{d-1}\nonumber\\
Y^dA_d&=&g_0(X)\cdots g_{d-2}(X).\nonumber
\end{eqnarray}

The last two equations tell us that $c^{-1}$ is a root in $\mathbb K$ of the polynomial $$U_{\bf w}(Y)=A_dY^d-A_{d-1}Y^{d-1}+\cdots+(-1)^{d-1}A_1Y+(-1)^d\in\mathbb K[Y].$$

For each such root we find a common solution in $\mathbb K$ of the first $d-2$ equations above, with $Y$ substituted by $c^{-1}$ (we can do this by computing the greatest common divisor of the corresponding $d-2$ polynomials, and factoring this common divisor completely). This common solution is a candidate for the primitive root of $\mathbb K$. As we will see in the example below, same projective codeword of minimum weight $[{\bf w}]$ can lead to different primitive roots.

\begin{exm} Let us consider the situation in Example \ref{attack_mincodewords}. We have

$$g_0(X)=-X^6, g_1(X)=X^3(1+X+X^2), g_2(X)=-X(1+X+X^2), g_3(X)=1.$$

$\bullet$ Let us consider the projective codeword of minimum weight, listed first in the table in that example: $[{\bf w}]=[5,5,1,1,0,0]$. We have ${\bf q}({\bf w})=(5,4,4,4)$, so $A_1=5, A_2=4, A_3=4, A_4=4$. The polynomial $$U_{\bf w}(Y)=4Y^4-4Y^3+4Y^2-5Y+1=(-3)(Y-1)^2(Y-3)^2,$$ has roots $c^{-1}=1,3$.

The greatest common divisor of the polynomials
\begin{eqnarray}
&&g_0(X)+g_1(X)+g_2(X)-5c^{-1}+1\nonumber\\
&&g_0(X)g_1(X)+g_0(X)g_2(X)+g_1(X)g_2(X)-4(c^{-1})^2+5c^{-1}-1\nonumber\\
&&g_0(X)g_1(X)g_2(X)-4(c^{-1})^3+4(c^{-1})^2-5c^{-1}+1\nonumber
\end{eqnarray} is the constant polynomial 1 for both values of $c^{-1}$; so no common solution.

$\bullet$ Let us consider the projective codeword of minimum weight, listed third in the table in that example: $[{\bf w}]=[3,0,1,6,0,1]$. We have ${\bf q}({\bf w})=(4,2,3,4)$, so $A_1=4, A_2=2, A_3=3, A_4=4$. The polynomial $$U_{\bf w}(Y)=4Y^4-3Y^3+2Y^2-4Y+1=(-3)(Y-1)^2(Y+2)(Y+1),$$ has roots $c^{-1}=1,-2,-1$.

The greatest common divisor of the polynomials
\begin{eqnarray}
&&g_0(X)+g_1(X)+g_2(X)-4c^{-1}+1\nonumber\\
&&g_0(X)g_1(X)+g_0(X)g_2(X)+g_1(X)g_2(X)-2(c^{-1})^2+4c^{-1}-1\nonumber\\
&&g_0(X)g_1(X)g_2(X)-3(c^{-1})^3+2(c^{-1})^2-4c^{-1}+1\nonumber
\end{eqnarray} is

\begin{enumerate}
  \item $Y-3$, if $c^{-1}=1$,
  \item the constant polynomials $1$ if $c^{-1}=-2$,
  \item $Y+2$, if $c^{-1}=-1$.
\end{enumerate}

In the first case, $\alpha=3$, and $c=1$, which is exactly what we obtained in the table.

In the second case we have no common solution.

In the third case, $\alpha=5$, and $c=6$. We have ${\bf g}_5=(6,4,6,1,0,0)$, and $c{\bf g}_5=(1,3,1,6,0,0)$, which is indeed a permutation of ${\bf w}=(3,0,1,6,0,1)$.
\end{exm}

%%%%%%%%%%%%%%%%%%%%%%%%%%%%%%%%%%%%%%%%%%%%%%%%%%%%%%%%
% Back to single space
\renewcommand{\baselinestretch}{1.0}
\small\normalsize % to get previous line to take
%%%%%%%%%%%%%%%%%%%%%%%%%%%%%%%%%%%%%%%%%%%%%%%%%%%%%%%%

\bibliographystyle{amsalpha}

\end{document}